\documentclass{amsart}
\usepackage{amsfonts,amssymb,amsmath,amsthm,mathrsfs}
\usepackage{url}
\usepackage{enumerate}

\newtheorem{theorem}{Theorem}[section]
\newtheorem{lemma}[theorem]{Lemma}

\newtheorem{corollary}[theorem]{Corollary}
\theoremstyle{definition}

\newtheorem{remark}[theorem]{Remark}

\numberwithin{equation}{section}

\author[G. Hu]{Guoen Hu}
\address{Guoen Hu, School of Applied Mathematics, Beijing Normal University, Zhuhai 519087,
P. R. China}
\email{guoenxx@163.com}
\thanks{}

\begin{document}

\title[Calder\'on-Zygmund operator]{Weighted weak type endpoint estimates for the compositions of Calder\'on-Zygmund operators}

\begin{abstract}
Let $T_1$, $T_2$ be two Calder\'on-Zygmund operators and $T_{1,\,b}$ be the commutator of $T_1$ with symbol $b\in {\rm BMO}(\mathbb{R}^n)$. In this paper, by establishing bilinear sparse   dominations, the author proves that composite operator $T_1T_2$ satisfies the following estimate: for $\lambda>0$ and weight $w\in A_1(\mathbb{R}^n)$,
\begin{eqnarray*}&&w\big(\{x\in\mathbb{R}^n:\,|T_{1} T_2f(x)|>\lambda\}\big)\\
&&\quad\lesssim [w]_{A_1}[w]_{A_{\infty}}\log ({\rm e}+[w]_{A_{\infty}}\big)
\int_{\mathbb{R}^n}\frac{|f(x)|}{\lambda}\log \Big({\rm e}+\frac{|f(x)|}{\lambda}\Big)w(x)dx,\nonumber
\end{eqnarray*}
while the composite operator $T_{1,b}T_2$ satisfies that
\begin{eqnarray*}&&w\big(\{x\in\mathbb{R}^n:\,|T_{1,b} T_2f(x)|>\lambda\}\big)\\
&&\quad\lesssim [w]_{A_1}[w]_{A_{\infty}}\log^2 ({\rm e}+[w]_{A_{\infty}}\big)
\int_{\mathbb{R}^n}\frac{|f(x)|}{\lambda}\log^2 \Big({\rm e}+\frac{|f(x)|}{\lambda}\Big)w(x)dx.
\end{eqnarray*}
\end{abstract}

\maketitle

\section{Introduction}
We will work on $\mathbb{R}^n$, $n\geq 1$. Let $p\in [1,\,\infty)$ and $w$ be a nonnegative, locally integrable function on $\mathbb{R}^n$. We say that   $w\in A_{p}(\mathbb{R}^n)$ if the $A_p$ constant $[w]_{A_p}$ is finite, where
$$[w]_{A_p}:=\sup_{Q}\Big(\frac{1}{|Q|}\int_Qw(x)dx\Big)\Big(\frac{1}{|Q|}\int_{Q}w^{-\frac{1}{p-1}}(x)dx\Big)^{p-1}<\infty,\,\,\,p\in (1,\,\infty),$$
the  supremum is taken over all cubes in $\mathbb{R}^n$, and
$$[w]_{A_1}:=\sup_{x\in\mathbb{R}^n}\frac{Mw(x)}{w(x)},$$
For properties of $A_p(\mathbb{R}^n)$, we refer the reader to the  monograph  \cite{gra}. In the last two decades, considerable attention has been paid to the sharp weighted bounds with $A_p$ weights for the classical operators in harmonic analysis. A prototypical work in this area is
Buckley's paper \cite{bu}, in which it was proved that if $p\in (1,\,\infty)$ and $w\in A_{p}(\mathbb{R}^n)$, then the Hardy-Littlewood maximal operator $M$ satisfies
\begin{eqnarray}\label{equ:1.1}\|Mf\|_{L^{p}(\mathbb{R}^n,\,w)}\lesssim_{n,\,p}[w]_{A_p}^{\frac{1}{p-1}}\|f\|_{L^{p}(\mathbb{R}^n,\,w)}.\end{eqnarray}
Moreover, the estimate (\ref{equ:1.1}) is sharp in the sense that the exponent $1/(p-1)$ can not be replaced by a smaller one. Hyt\"onen and P\'erez \cite{hp} improved   estimate (\ref{equ:1.1}), and showed that
\begin{eqnarray}\label{equ:1.2}\|Mf\|_{L^{p}(\mathbb{R}^n,\,w)}\lesssim_{n,\,p}
\big([w]_{A_p}[w^{-\frac{1}{p-1}}]_{A_{\infty}}\big)^{\frac{1}{p}}\|f\|_{L^{p}(\mathbb{R}^n,\,w)},\end{eqnarray}
where and in the following, for a weight $u\in A_{\infty}(\mathbb{R}^n)=\cup_{p\geq 1}A_p(\mathbb{R}^n)$, $[u]_{A_{\infty}}$ is the $A_{\infty}$ constant of $u$, defined by
$$[u]_{A_{\infty}}=\sup_{Q\subset \mathbb{R}^n}\frac{1}{u(Q)}\int_{Q}M(u\chi_Q)(x)dx,$$
see \cite{wil}.

Let $K$ be a locally integrable function on $\mathbb{R}^n\times \mathbb{R}^n\backslash \{(x,\,y):\,x\not =y\}$. We say that $K$ is a Calder\'on-Zygumnd kernel, if $K$ satisfies the size condition that for $x,\,y\in\mathbb{R}^n$, $x\not =y$,
$$|K(x,\,y)|\lesssim |x-y|^{-n},\,\,\hbox{if}\,\,\,x\not =y,$$
and the regularity condition that for any $x,\,y,\,y'\in \mathbb{R}^n$ with $|x-y|\geq 2|y-y'|$,
\begin{eqnarray}\label{eq1.verepsilon}|K(x,\,y)-K(x,\,y')|+|K(y,\,x)-K(y',\,x)|\lesssim \frac{|y-y'|^{\varepsilon}}{|x-y|^{n+\varepsilon}},\end{eqnarray}
where $\varepsilon\in (0,\,1]$ is a constant. An  linear operator $T$ is said to be a Calder\'on-Zygmund operator with kernel $K$ if
it is bounded on $L^2(\mathbb{R}^n)$, and satisfies that
\begin{eqnarray}\label{equ:1.3}Tf(x)=\int_{\mathbb{R}^n}K(x,\,y)f(y)dy,\end{eqnarray}
for all $f\in L^2(\mathbb{R}^n)$ with compact support and a. e. $x\in\mathbb{R}^n\backslash {\rm supp}\, f$. Hyt\"onen \cite{hyt}
proved that  for a  Calder\'on-Zygmund operator $T$ and $w\in A_2(\mathbb{R}^n)$,
\begin{eqnarray}\label{equ:1.4}\|Tf\|_{L^{2}(\mathbb{R}^n,\,w)}\lesssim_{n}[w]_{A_2}\|f\|_{L^{2}(\mathbb{R}^n,\,w)}.\end{eqnarray}
This solved  the so-called $A_2$ conjecture. Hyt\"onen and Lacey \cite{hyla} improved the estimate (\ref{equ:1.4}) and proved that
for a Calder\'on-Zygmund operator $T$, $p\in (1,\,\infty)$ and $w\in A_p(\mathbb{R}^n)$,
\begin{eqnarray}\label{equ:1.5}\|Tf\|_{L^{p}(\mathbb{R}^n,\,w)}\lesssim_{n,\,p}[w]_{A_p}^{\frac{1}{p}} \big([w]_{A_{\infty}}^{\frac{1}{p'}}+[\sigma]_{A_{\infty}}^{\frac{1}{p}}\big)\|f\|_{L^{p}(\mathbb{R}^n,\,w)}.\end{eqnarray}
Here and in the following, for $p\in (1,\,\infty)$ and $w\in A_p(\mathbb{R}^n)$, $p'=p/(p-1)$, $\sigma=w^{-\frac{1}{p-1}}$.  Hyt\"onen and P\'erez \cite{hp2}  proved that
if  $T$ is a Calder\'on-Zygmund operator and $w\in A_1(\mathbb{R}^n)$, then
$$\|Tf\|_{L^{1,\,\infty}(\mathbb{R}^n,\,w)}\lesssim [w]_{A_1}\log ({\rm e}+[w]_{A_{\infty}}\big)
\|f\|_{L^1(\mathbb{R}^n,\,w)}.$$
For other  works about the quantitative weighted bounds for Calder\'on-Zygmund operators, see \cite{hlp,hp,ler1, ler2, ler3} and the related references therein.

Let $T_1$, $T_2$ be two Calder\'on-Zygmund operators, and $T_2^*$ be the adjoint operator of $T_2$. It was pointed out in \cite[Section 9]{cm97} that if $T_1(1)=T_2^*(1)=0$, then the composite operator $T_1 T_2$ is also a Calder\'on-Zygmund operator, thus for $p\in (1,\,\infty)$ and $w\in A_p(\mathbb{R}^n)$,
$$\|T_1 T_2f\|_{L^{p}(\mathbb{R}^n,\,w)}\lesssim [w]_{A_p}^{\frac{1}{p}} \big([w]_{A_{\infty}}^{\frac{1}{p'}}+[\sigma]_{A_{\infty}}^{\frac{1}{p}}\big)\|f\|_{L^{p}(\mathbb{R}^n,\,w)}.
$$
Benea and Bernicot \cite{bb} considered the weighted bounds for $T_1 T_2$ when $T_1(1)$ or $T_2^*(1)=0$. In fact, the results in \cite{bb}  implies the following conclusion (see  Remark \ref{r3.2} in Section 3).
\begin{theorem}\label{thm1.0}
Let $T_1$, $T_2$ be two Calder\'on-Zygmund operators, $p\in (1,\,\infty)$ and $w\in A_p(\mathbb{R}^n)$.
\begin{itemize}
\item[\rm (i)] If $T_1(1)=0$, then
$$\|T_1 T_2f\|_{L^{p}(\mathbb{R}^n,\,w)}\lesssim_{n,\,p}[w]_{A_p}^{\frac{1}{p}} \big([w]_{A_{\infty}}^{\frac{1}{p'}}+[\sigma]_{A_{\infty}}^{\frac{1}{p}}\big)[\sigma]_{A_{\infty}}\|f\|_{L^{p}(\mathbb{R}^n,\,w)};
$$
\item[\rm (ii)] if $T_2^*(1)=0$, then
$$\|T_1 T_2f\|_{L^{p}(\mathbb{R}^n,\,w)}\lesssim_{n,\,p}[w]_{A_p}^{\frac{1}{p}} \big([w]_{A_{\infty}}^{\frac{1}{p'}}+[\sigma]_{A_{\infty}}^{\frac{1}{p}}\big)[w]_{A_{\infty}}\|f\|_{L^{p}(\mathbb{R}^n,\,w)}.
$$
\end{itemize}
\end{theorem}
Fairly recently, Hu \cite{hu1} considered the quantitative weighted bounds for the composite operator $T_{1,b}T_2$, with $T_1,\,T_2$ two Calder\'on-Zygmund operators, $b\in {\rm BMO}(\mathbb{R}^n)$, $T_{1,b}$  the commutator of $T_1$  defined by
$$T_{1,b}f(x)=b(x)T_1f(x)-T_1(bf)(x),
$$see \cite{cpp,hp,ler3} for the quantitative weighted bounds of commutator of Calder\'on-Zygmund operators. Employing the ideas of Lerner \cite{ler4},  Hu \cite{hu1} proved that \begin{eqnarray}\label{equ:1.8}\|T_{1,b}  T_2f\|_{L^p(\mathbb{R}^n,\,w)}&\lesssim &\|b\|_{{\rm BMO}(\mathbb{R}^n)}[w]_{A_p}^{\frac{1}{p}} \big([w]_{A_{\infty}}^{\frac{1}{p'}}+[\sigma]_{A_{\infty}}^{\frac{1}{p}}\big)\\
&&\quad\times\big([w]_{A_{\infty}}+[\sigma]_{A_{\infty}}\big)^2
\big\|f\big\|_{L^p(\mathbb{R}^n,\,w)}.\nonumber
\end{eqnarray}
We remark that the operator $T_{1,b}T_2$ was introduced by Krantz and Li \cite{krali} in the study of the Toeplitz type operator
of singular integral operators.

The main purpose of this paper is to establish the weighted weak type endpoint estimate for the composite operators $T_1 T_2$ and  $T_{1,\,b} T_2$. Our main result can be stated as follows.

\begin{theorem}\label{thm1.1}Let $T_1$ and $T_2$  be Calder\'on-Zygmund operators.  Then
for $w\in A_{1}(\mathbb{R}^n)$ and $\lambda>0$,
\begin{eqnarray}\label{eq1.8}&&w\big(\{x\in\mathbb{R}^n:\,|T_{1} T_2f(x)|>\lambda\}\big)\\
&&\quad\lesssim [w]_{A_1}[w]_{A_{\infty}}\log ({\rm e}+[w]_{A_{\infty}}\big)
\int_{\mathbb{R}^n}\frac{|f(x)|}{\lambda}\log \Big({\rm e}+\frac{|f(x)|}{\lambda}\Big)w(x)dx.\nonumber
\end{eqnarray}
Moreover, if $T_1(1)=0$, then
\begin{eqnarray}\label{eq1.81}&&w\big(\{x\in\mathbb{R}^n:\,|T_{1} T_2f(x)|>\lambda\}\big)\\
&&\quad\lesssim [w]_{A_1}\log^2 ({\rm e}+[w]_{A_{\infty}}\big)
\int_{\mathbb{R}^n}\frac{|f(x)|}{\lambda}\log \Big({\rm e}+\frac{|f(x)|}{\lambda}\Big)w(x)dx.\nonumber
\end{eqnarray}
\end{theorem}
\begin{theorem}\label{thm1.2}Let $T_1$ and $T_2$  be Calder\'on-Zygmund operators, $b\in {\rm BMO}(\mathbb{R}^n)$.  Then
for $w\in A_{1}(\mathbb{R}^n)$ and $\lambda>0$,
\begin{eqnarray}\label{eq1.9}&&w\big(\{x\in\mathbb{R}^n:\,|T_{1,b} T_2f(x)|>\lambda\}\big)\\
&&\quad\lesssim [w]_{A_1}[w]_{A_{\infty}}^2\log ({\rm e}+[w]_{A_{\infty}}\big)
\int_{\mathbb{R}^n}\frac{|f(x)|}{\lambda}\log^2\Big({\rm e}+\frac{|f(x)|}{\lambda}\Big)w(x)dx.\nonumber
\end{eqnarray}
\end{theorem}

Throughout the article,  $C$ always denotes a
positive constant that is independent of the main parameters
involved but whose value may differ from line to line. We use the
symbol $A\lesssim B$ to denote that there exists a universal constant
$C$ such that $A\le CB$.  Specially, we use $A\lesssim_{n,p} B$ to denote that there exists a positive constant
$C$ depending only on $n,p$ such that $A\le CB$. Constant with subscript such as $c_1$,
does not change in different occurrences. For any set $E\subset\mathbb{R}^n$,
$\chi_E$ denotes its characteristic function.  For a cube
$Q\subset\mathbb{R}^n$ and $\lambda\in(0,\,\infty)$, we use
$\lambda Q$ to denote the cube with the same center as $Q$ and whose
side length is $\lambda$ times that of $Q$.
For $\beta\in [0,\,\infty)$,   cube $Q\subset \mathbb{R}^n$ and a suitable function $g$, $\|g\|_{L(\log L)^{\beta},\,Q}$ is   defined by
$$\|g\|_{L(\log L)^{\beta},\,Q}=\inf\Big\{\lambda>0:\,\frac{1}{|Q|}\int_{Q}\frac{|g(y)|}{\lambda}\log^{\beta}\Big({\rm e}+\frac{|g(y)|}{\lambda}\Big)dy\leq 1\Big\}.$$
We denote $\|g\|_{L(\log L)^{0},\,Q}$ by $\langle |g|\rangle_{Q}$. For $r\in (0,\,\infty)$, we set $\langle |g|\rangle_{r,Q}=(\langle|g|^r\rangle_{Q})^{1/r}$. For $\beta\in [0,\,\infty)$, the maximal operator $M_{L(\log L)^{\beta}}$ is defined by
$$M_{L(\log L)^{\beta}}f(x)=\sup_{Q\ni x}\|f\|_{L(\log L)^{\beta},\,Q}.$$
It is well known that
\begin{eqnarray}\label{eq1.10}
\big|\{x\in\mathbb{R}^n:\, M_{L(\log L)^{\beta}}f(x)>\lambda\}\big|\lesssim \int_{\mathbb{R}^n}\frac{|f(x)|}{\lambda}
\log ^{\beta}\Big({\rm e}+\frac{|f(x)|}{\lambda}\Big)dx.
\end{eqnarray}
\section{Two grand maximal operators}
For a linear operator $T$, we define the corresponding  grand maximal  operator $\mathcal{M}_T$  by
$$\mathcal{M}_{T}f(x)=\sup_{Q\ni x}{\rm ess}\sup_{\xi\in Q}|T(f\chi_{\mathbb{R}^n\backslash 3Q})(\xi)|,$$
where the supremum is taken over all cubes $Q\subset \mathbb{R}^n$ containing $x$.
This operator was introduced by Lerner \cite{ler2}. Moreover, Lerner \cite{ler2} proved that a Calder\'on-Zygmund operator $T$ with kernel $K$ in the sense of (\ref{equ:1.3}) satisfies that
\begin{eqnarray}\label{eq2.1}\mathcal{M}_Tf(x)\lesssim T^*f(x)+Mf(x).\end{eqnarray}
where $T^*$ denotes the maximal operator associated with $T$, defined by
$$T^*f(x)=\sup_{\epsilon>0}\Big|\int_{|x-y|>\epsilon}K(x,\,y)f(y)dy\Big|.$$Let $T_1$, $T_2$ be Calder\'on-Zygmund operators and $b\in {\rm BMO}(\mathbb{R}^n)$. We define the grand maximal operators $\mathcal{M}_{T_1T_2}^{*}$ and $\mathcal{M}_{T_1T_{2,b}}^{*}$ by
$$\mathcal{M}_{T_1T_2}^{*}f(x)=\sup_{Q\ni x}{\rm ess}\sup_{\xi\in Q}|T_1\big(\chi_{\mathbb{R}^n\backslash 3Q}
T_2(f\chi_{\mathbb{R}^n\backslash 9Q})\big)(\xi)|$$
and
$$\mathcal{M}_{T_1T_{2, b}}^{*}f(x)=\sup_{Q\ni x}{\rm ess}\sup_{\xi\in Q}|T_1\big(\chi_{\mathbb{R}^n\backslash 3Q}
T_{2,b}(f\chi_{\mathbb{R}^n\backslash 9Q})\big)(\xi)|$$
respectively. As we will see in Section 3, these two operators play important roles in the proof of Theorem \ref{thm1.1}.
This section is devoted to the endpoint estimates for the operators $\mathcal{M}_{T_1T_2}^{*}$ and $\mathcal{M}_{T_1T_{2, b}}^{*}$. We begin with some preliminary lemmas.
\begin{lemma}\label{lem2.1}
Let $p_0\in (1,\,\infty)$, $\varrho\in [0,\,\infty)$ and $S$ be a sublinear operator. Suppose that
$$\|Sf\|_{L^{p_0}(\mathbb{R}^n)}\le A_1 \|f\|_{L^{p_0}(\mathbb{R}^n)},$$
and for all $\lambda>0$,
$$\big|\{x\in\mathbb{R}^n:\,|Sf(x)|>\lambda\}\big|\le A_2 \int_{\mathbb{R}^n}\frac{|f(x)|}{\lambda}
\log ^{\varrho}\Big({\rm e}+\frac{|f(x)|}{\lambda}\Big)dx.$$
Then for $\beta\in [0,\,\infty)$ and  two  cubes $Q_2,\,Q_1\subset \mathbb{R}^n$,
\begin{eqnarray}\label{equ2.2}&&\int_{Q_1}|S(f\chi_{Q_2})(x)|\log^{\beta}\big({\rm e}+|S(f\chi_{Q_2})(x)|\big)dx\\
&&\quad\lesssim |Q_1|+(A_1^{p_0}+A_2)\int_{Q_2}|f(x)|\log^{\beta+\varrho+1} ({\rm e}+|f(x)|) dx.\nonumber\end{eqnarray}
Moreover, if $\{Q_j\}\subset \mathcal{D}(Q_0)$ with pairwise disjoint interiors, then
\begin{eqnarray}\label{equ2.3}
\sum_{j}\int_{5Q_j}|S(f\chi_{Q_j})(x)|dx\lesssim (1+A_1^{p_0}+A_2)|Q_0|\|f\|_{L(\log L)^{\varrho+1},\,Q_0}.
\end{eqnarray}
\end{lemma}
For the case of $\beta=0$, (\ref{equ2.2}) was proved in \cite{huyang}. For the case of $\beta\in (0,\,\infty)$, the proof is similar.
And (\ref{equ2.3}) follows from (\ref{equ2.2}) by homogeneity.

\begin{lemma}\label{lem2.2}
Let $s\in [0,\,\infty)$, $S$ be a sublinear operator which satisfies that for any $\lambda>0$,
$$\big|\{x\in \mathbb{R}^n:\, |Sf(x)|>\lambda\}\big|\lesssim \int_{\mathbb{R}^n}\frac{|f(x)|}{\lambda}\log^s\Big({\rm e}+\frac{|f(x)|}{\lambda}\Big)dx.
$$
Then for any $\varrho\in (0,\,1)$ and cube $Q\subset \mathbb{R}^n$,$$
\Big(\frac{1}{|Q|}\int_{Q}|S(f\chi_{Q})(x)|^{\varrho}dx\Big)^{\frac{1}{\varrho}}\lesssim \|f\|_{L(\log L)^{s},\,Q}.$$
\end{lemma}

For the proof of Lemma \ref{lem2.2}, see \cite[p. 643]{huli}.

\begin{lemma}\label{lem2.3}
Let $R>1$, $\Omega\subset \mathbb{R}^n$ be a open set. Then  $\Omega$ can be decomposed as
$\Omega=\cup_{j}Q_j$, where $\{Q_j\}$ is a sequence of cubes with disjoint interiors, and
\begin{itemize}
\item[\rm (i)]
$$5R\le \frac{{\rm dist}(Q_j,\,\mathbb{R}^n\backslash \Omega)}{{\rm diam} Q_j}\le 15R,$$
\item[\rm (ii)] $\sum_{j}\chi_{RQ_j}(x)\lesssim_{n,R} \chi_{\Omega}(x).$
\end{itemize}
\end{lemma}
For the proof of Lemma \ref{lem2.3}, see \cite[p.\,256]{saw}.

\begin{lemma}\label{lem2.4} Let $\beta\in [0,\,\infty)$,  $U$ be a sublinear operator which is bounded on $L^2(\mathbb{R}^n)$, and satisfies that for any $t>0$,
$$\big|\{x\in\mathbb{R}^n:\, |Uf(x)>t\}\big|\lesssim \int_{\mathbb{R}^n}\frac{|f(x)|}{t}\log ^{\beta}\Big({\rm e}+\frac{|f(x)|}{t}\Big)dx.$$
Let $T$ be a Calder\'on-Zygmund operator and $b\in {\rm BMO}$ with $\|b\|_{{\rm BMO}(\mathbb{R}^n)}=1$. Then for any $\lambda>0$,
\begin{eqnarray}\label{eq2.4}\big|\{x\in\mathbb{R}^n:\, |UTf(x)|>\lambda\}\big|\lesssim
\int_{\mathbb{R}^n}\frac{|f(x)|}{\lambda}\log^{\beta+1} \Big ({\rm e}+\frac{|f(x)|}{\lambda}\Big)dx;\end{eqnarray}
and
\begin{eqnarray}\label{eq2.5}\big|\{x\in\mathbb{R}^n:\, |UT_bf(x)|>\lambda\}\big|\lesssim
\int_{\mathbb{R}^n}\frac{|f(x)|}{\lambda}\log^{\beta+2} \Big ({\rm e}+\frac{|f(x)|}{\lambda}\Big)dx.\end{eqnarray}
\end{lemma}
\begin{proof}We only prove   (\ref{eq2.5}). The proof of   (\ref{eq2.4}) is similar, simpler, and will be omitted.
By homogeneity, it suffices to prove  the inequality (\ref{eq2.5}) for the case of $\lambda=1$.
Applying Lemma \ref{lem2.3} to the set $\{x\in\mathbb{R}^n:\,Mf(x)>1\}$,   we  obtain a sequence of cubes $\{Q_l\}$ with disjoint interiors, such that $\{x\in\mathbb{R}^n:\,Mf(x)>1\}=\cup_{l}Q_l$, and
$$\sum_{l}\chi_{5Q_l}(x)\lesssim 1,\,\,\, \int_{Q_l}|f(y)|dy\lesssim |Q_l|.$$
Let $$g(x)=f(x)\chi_{\mathbb{R}^n\backslash \cup_lQ_l}(x)+\sum_{l}\langle f\rangle_{Q_l}\chi_{Q_l}(x),$$
and
$$h(x)=\sum_{l}(f(x)-\langle f\rangle_{Q_l})\chi_{Q_l}(x):=\sum_{l}h_l(x).$$
Recall that $UT_b$ is bounded on $L^2(\mathbb{R}^n)$. Thus by the fact that $\|g\|_{L^{\infty}(\mathbb{R}^n)}\lesssim 1$, we get  that
\begin{eqnarray*}\big|\{x\in\mathbb{R}^n:\,|UT_bg(x)|>1/2\big\}\big|\lesssim\int_{\mathbb{R}^n}|f(x)|^{2}dx\lesssim \int_{\mathbb{R}^n}|f(x)|dx.\end{eqnarray*}

To estimate $UT_bh$,  write
\begin{eqnarray*}|UT_bh(x)|&\le & \Big|UT\Big(\sum_{l}\big(b-\langle b\rangle_{Q_l}\big)h_l\Big)(x)\Big|+\Big|U\Big(\sum_{l}\big(b-b\rangle_{Q_l}\big)\chi_{5Q_l}Th_l\Big)(x)\Big|\\
&&+\Big|U\Big(\sum_{l}\chi_{\mathbb{R}^n\backslash 5Q_l}\big(b-\langle b\rangle_{Q_l}\big)Th_l\Big)(x)\Big|\\
&=&{\rm V}_1(x)+{\rm V}_2(x)+{\rm V}_3(x).\nonumber
\end{eqnarray*}
We first consider the term ${\rm V}_1$.  Employing Jensen's inequality, we have that for $\gamma>0$,
$$\langle|f|\rangle_Q\log^{\gamma}\big({\rm e}+\langle|f|\rangle_Q)\le \frac{1}{|Q|}\int_{Q}|f(y)|\log^{\gamma} ({\rm e}+|f(y)|)dy.
$$
Observe that for $t_1,\,t_2\in (0,\,\infty)$,
$$(t_1+t_2)\log^{\gamma} ({\rm e}+t_1+t_2)\lesssim_{\gamma}\big[t_1\log^{\gamma} ({\rm e}+t_1)+t_2\log ({\rm e}+t_2)\big].$$
It then follows that
\begin{eqnarray}\label{eq2.6}
\int_{Q_l}|h_l(y)|\log^{\gamma} ({\rm e}+|h_l(y)|)dy\lesssim \int_{Q_l}|f(y)|\log^{\gamma} ({\rm e}+|f(y)|)dy.
\end{eqnarray}
On the other hand, the generalization of H\"older inequality (see \cite{rr}) tells  us that
\begin{eqnarray}\label{eq2.7}t_1t_2\log^{\beta} ({\rm e}+t_1t_2)\lesssim {\rm exp}t_1+\log^{\beta+1}({\rm e}+t_2).\end{eqnarray}
We deduce from inequalities (\ref{eq2.4}), (\ref{eq2.6}) and (\ref{eq2.7}) that
\begin{eqnarray}\label{eq2.8}
&&\big|\{x\in\mathbb{R}^n: |{\rm V}_1(x)|>1/6\}\big|\\
&&\quad\lesssim\sum_{l}\int_{Q_l}|b(x)-\langle b\rangle_{Q_l}||h_l(x)|
\log^{\beta+1} ({\rm e}+|b(x)-\langle b\rangle_{Q_l}||h_l(x)|)dx\nonumber\\
&&\quad\lesssim\sum_{l}\int_{Q_l}{\rm exp}\Big(\frac{|b(x)-\langle b\rangle_{Q_l}|}{C\|b\|_{{\rm BMO}(\mathbb{R}^n)}}\Big)dx+\sum_{l}\int_{Q_l}|h_l(x)|\log^{\beta+2}({\rm e}+|h_l(x)|)dx\nonumber\\
&&\quad\lesssim \int_{\mathbb{R}^n}|f(x)|\log^{\beta+2} ({\rm e}+|f(x)|)dx.\nonumber\end{eqnarray}
Recall that $\chi_{\cup_{l}5Q_l}\lesssim1$. It follows from  Lemma \ref{lem2.1}, inequalities (\ref{eq2.6}) and (\ref{eq2.7}), that
\begin{eqnarray}\label{eq2.9}&&\big|\{x\in\mathbb{R}^n: |{\rm V}_2(x)|>1/6\}\big|\\
&&\quad\lesssim\sum_l\int_{5Q_l}|b(x)-\langle b\rangle_{Q_l}||Th_l(x)|\log^{\beta} ({\rm e}+|b(x)-\langle b\rangle_{Q_l}||Th_l(x)| )dx\nonumber\\
&&\quad\lesssim\sum_{l}\Big(|Q_l|+\int_{Q_l}|Th_l(y)|\log^{\beta+1} ({\rm e}+|Th_l(y)|)dy\Big)\nonumber\\
&&\quad\lesssim\int_{\mathbb{R}^n}|f(y)|\log^{\beta+2} ({\rm e}+|f(y)|)dy.\nonumber\end{eqnarray}

To estimate term ${\rm V}_3$, we first observe that for each $l$ and  $y\in\mathbb{R}^n\backslash 5Q_l$,
\begin{eqnarray}\label{eq2.oper}|Th_l(y)|\lesssim \frac{\{\ell(Q_l)\}^{\varepsilon}}{|y-z_l|^{n+\varepsilon}}\|h_l\|_{L^1(\mathbb{R}^n)},\end{eqnarray}
here $z_l$ is the center of $Q_l$, $\varepsilon\in (0,\,1]$ is a constant.
Thus, for each
$v\in L^{2}(\mathbb{R}^n)$ with  $\|v\|_{L^2(\mathbb{R}^n)}=1$, we have by the John-Nirenberg inequality that
\begin{eqnarray*}
&&\sum_{l}\Big|\int_{\mathbb{R}^n\backslash 5Q_l}\big(b(y)-\langle b\rangle_{Q_l}\big)Th_l(y)v(y)dy\Big|\\
&&\quad\lesssim\sum_{l}\int_{\mathbb{R}^n}|h_l(z)|\int_{\mathbb{R}^n\backslash 5Q_l}\frac{|b(y)-\langle b\rangle_{Q_l}||Q_l|^{\varepsilon/n}}{|y-z_l|^{n+\varepsilon}}|v(y)|dydz\\
&&\quad\lesssim \sum_{l}\int_{Q_l}M_{L\log L}v(y)dy\lesssim\Big(\sum_l|Q_l|\Big)^{\frac{1}{2}}.
\end{eqnarray*}
This, via a standard duality argument, gives that
\begin{eqnarray}\label{eq2.10}\big|\{x\in\mathbb{R}^n:|{\rm V}_3(x)|>\frac{1}{6}\}\big|&\lesssim&\Big\|\sum_{l}\chi_{\mathbb{R}^n\backslash 5Q_l}\big(b(\cdot)-\langle b\rangle_{Q_l}\big)Th_l\Big\|_{L^2(\mathbb{R}^n)}^2\\
&\lesssim&\int_{\mathbb{R}^n}|f(y)|dy.\nonumber\end{eqnarray}
Combining the estimates (\ref{eq2.8})-(\ref{eq2.9}) and (\ref{eq2.10}) leads to our desired conclusion.
\end{proof}
For a Calder\'on-Zygmund operator $T$ and $b\in {\rm BMO}(\mathbb{R}^n)$, let $T^*_b$ be the maximal commutator defined by
$$T^*_bf(x)=\sup_{\epsilon>0}\Big|\int_{|x-y|>\epsilon}\big(b(x)-b(y)\big)K(x,\,y)f(y)dy\Big|,
$$
and $M_b$ the commutator defined by
$$M_bf(x)=\sup_{Q\ni x}\frac{1}{|Q|}\int_{Q}|b(x)-b(y)||f(y)|dy.$$

\begin{corollary}\label{co2.1} Let $T_1$ and $T_2$ be two Calder\'on-Zygmund operators. Then for  each $\lambda>0$, \begin{eqnarray}\label{equ2.10}&&\big|\{x\in\mathbb{R}^n:MT_2f(x)+T_1^*T_2f(x)>\lambda\}\big|\lesssim \int_{\mathbb{R}^n}\frac{|f(x)|}{\lambda}
\log\Big({\rm e}+\frac{|f(x)|}{\lambda}\Big)dx,\end{eqnarray}\begin{eqnarray}\label{eq2.12}\big|\{x\in\mathbb{R}^n:\,MT_{2,b}f(x)>\lambda\}\big|\lesssim \int_{\mathbb{R}^n}\frac{|f(x)|}{\lambda}
\log^2\Big({\rm e}+\frac{|f(x)|}{\lambda}\Big)dx,\end{eqnarray}
and for $r\in (0,\,1)$,
\begin{eqnarray}\label{eq2.11}&&\big|\{x\in\mathbb{R}^n:\,M_rT_1^*T_{2,b}f(x)>\lambda\}\big|\lesssim \int_{\mathbb{R}^n}\frac{|f(x)|}{\lambda}
\log^2\Big({\rm e}+\frac{|f(x)|}{\lambda}\Big)dx.\end{eqnarray}
\end{corollary}
\begin{proof}
The inequalities (\ref{equ2.10}) and (\ref{eq2.12}) follows from Lemma \ref{lem2.4} directly. Also, we know from Lemma \ref{lem2.4}  that
\begin{eqnarray*}&&\big|\{x\in\mathbb{R}^n:\,T_{1}^*T_{2,b}f(x)>\lambda\}\big|\lesssim \int_{\mathbb{R}^n}\frac{|f(x)|}{\lambda}
\log^2\Big({\rm e}+\frac{|f(x)|}{\lambda}\Big)dx.\end{eqnarray*}Recall that for $r\in (0,\,1)$,
\begin{eqnarray}\label{eq2.xx}\big|\{x\in\mathbb{R}^n:
M_{r}f(x)>\lambda\}\big|\lesssim\lambda^{-1}\sup_{s\geq 2^{-\frac{1}{r}}\lambda}s\big|\{x\in\mathbb{R}^n:
|f(x)|>s\}|.
\end{eqnarray}
see \cite[p. 651]{huli}. Combining the last two inequalities
establishes (\ref{eq2.11}).
\end{proof}

We are now ready to establish our main conclusion in this section.
\begin{theorem}\label{thm2.1}
Let $T_1$, $T_2$ be two Calder\'on-Zygmund operators, $b\in{\rm BMO}(\mathbb{R}^n)$ with $\|b\|_{{\rm BMO}(\mathbb{R}^n)}=1$. Then for each bounded function $f$ with compact support and $\lambda>0$,
\begin{eqnarray}\label{eq2.13}
\big|\{x\in\mathbb{R}^n:\,\mathcal{M}_{T_1 T_2}^{*}f(x)>\lambda\}\big|\lesssim \int_{\mathbb{R}^n}\frac{|f(x)|}{\lambda}
\log\Big({\rm e}+\frac{|f(x)|}{\lambda}\Big)dx,
\end{eqnarray}
and
\begin{eqnarray}\label{eq2.14}
\big|\{x\in\mathbb{R}^n:\,\mathcal{M}_{T_1 T_{2,b}}^{*}f(x)>\lambda\}\big|\lesssim \int_{\mathbb{R}^n}\frac{|f(x)|}{\lambda}
\log^2\Big({\rm e}+\frac{|f(x)|}{\lambda}\Big)dx.
\end{eqnarray}
\end{theorem}
\begin{proof} We first consider the estimate (\ref{eq2.13}). By Corollary \ref{co2.1}, it suffices to prove that
\begin{eqnarray}\label{eq2.15}\mathcal{M}_{T_1 T_2}^{*}f(x)\lesssim M_{\frac{1}{2}}T_1^*T_2f(x)+M_{L\log L}f(x)+MT_2f(x).\end{eqnarray}
Let $x\in \mathbb{R}^n$ and $Q$ be a cube containing $x$.
A trivial computation involving (\ref{eq2.1}) leads to that for each $\xi\in Q$,
\begin{eqnarray*}
&&T_1\big(\chi_{\mathbb{R}^n\backslash 9Q}T_2(f\chi_{\mathbb{R}^n\backslash 9Q})\big)(\xi)\lesssim
\inf_{z\in Q}\mathcal{M}_{T_1}(T_2f\chi_{\mathbb{R}^n\backslash 9Q})(z)\\
&&\quad\lesssim\Big(\frac{1}{|Q|}\int_{Q}\Big(\mathcal{M}_{T_1}\big(T_2(f\chi_{\mathbb{R}^n\backslash 9Q})\big)(z)\Big)^{\frac{1}{2}}dz\Big)^2\\
&&\quad\lesssim\Big(\frac{1}{|Q|}\int_{Q}\big[T_1^*T_2f(z)\big]^{\frac{1}{2}}dz\Big)^2
+\Big(\frac{1}{|Q|}\int_{Q}\big[MT_2f(z)\big]^{\frac{1}{2}}dz\Big)^2\\
&&\qquad+\Big(\frac{1}{|Q|}\int_{Q}\big[T_1^*T_2(f\chi_{9Q})(\xi)\big]^{\frac{1}{2}}d\xi\Big)^2+\Big(\frac{1}{|Q|}
\int_{Q}\big[M(T_2f\chi_{9Q})(\xi)\big]^{\frac{1}{2}}d\xi\Big)^2.
\end{eqnarray*}
Recalling that $M_{\frac{1}{2}}Mh(x)\lesssim Mh(x)$, we know that
\begin{eqnarray*}
\Big(\frac{1}{|Q|}\int_{Q}\big(MT_2f(\xi)\big)^{\frac{1}{2}}d\xi\Big)^2\lesssim MT_2f(x).
\end{eqnarray*}
On the other hand, it follows from Lemma \ref{lem2.2} that
\begin{eqnarray*}
&&\Big(\frac{1}{|Q|}\int_{Q}\big[T_1^*T_2(f\chi_{9Q})(\xi)\big]^{\frac{1}{2}}d\xi\Big)^2
+\Big(\frac{1}{|Q|}\int_{Q}\big[M(T_2f\chi_{9Q})(\xi)\big]^{\frac{1}{2}}d\xi\Big)^2\\
&&\quad \lesssim\|f\|_{L\log L,\,9Q}\lesssim M_{L\log L}f(x).
\end{eqnarray*}
This establishes (\ref{eq2.15}).

We turn our attention to the inequality (\ref{eq2.14}). Again by Corollary \ref{co2.1}, it suffices to prove that
\begin{eqnarray}\label{eq2.19}\mathcal{M}_{T_1 T_{2,b}}^{*}f(x)\lesssim M_{\frac{1}{2}}T_1^*T_{2,b}f(x)+MT_{2,b}f(x)+M_{L(\log L)^2}f(x).\end{eqnarray}
Let $x\in \mathbb{R}^n$ and $Q$ be a cube containing $x$.
As in the proof of (\ref{eq2.15}), we have that for each $\xi\in Q$,
\begin{eqnarray*}
&&T_1\big(\chi_{\mathbb{R}^n\backslash 3Q}T_{2,b}(f\chi_{\mathbb{R}^n\backslash 9Q})\big)(\xi)\\
&&\quad\lesssim\Big(\frac{1}{|Q|}\int_{Q}\Big[\mathcal{M}_{T_1}\big(T_{2,b}(f\chi_{\mathbb{R}^n\backslash 9Q})\big)(z)\Big]^{\frac{1}{2}}dz\Big)^2\\
&&\quad\lesssim\Big(\frac{1}{|Q|}\int_{Q}\Big[T_1^*\big(T_{2,b}(f\chi_{\mathbb{R}^n\backslash 9Q})\big)(z)\Big]^{\frac{1}{2}}dz\Big)^2\\
&&\qquad+\Big(\frac{1}{|Q|}\int_{Q}\Big[M\big(T_{2,b}(f\chi_{\mathbb{R}^n\backslash 9Q})\big)(z)\Big]^{\frac{1}{2}}dz\Big)^2={\rm I}+{\rm II}.
\end{eqnarray*}
A trivial computation involving the John-Nirenberg inequality and Lemma \ref{lem2.2} shows  that
\begin{eqnarray*}
{\rm I}&\lesssim&\Big(\frac{1}{|Q|}\int_{Q}\big[T_{1}^*T_{2,b}f(z)\big]^{\frac{1}{2}}dz\Big)^2
+\Big(\frac{1}{|Q|}\int_{Q}|T_1^* T_{2,b}(f\chi_{9Q})(z)|^{\frac{1}{2}}dz\Big)^2\\
&\lesssim&M_{\frac{1}{2}}T_{1}^*T_{2,b}f(x)+M_{L(\log L)^2}f(x).
\end{eqnarray*}
Similarly, we can obtain that
$${\rm II}\lesssim MT_{2,b}f(x)+M_{L(\log L)^2}f(x).
$$
Combining the estimates above yields (\ref{eq2.19}).
\end{proof}

\section{Proof of Theorem \ref{thm1.1}}
Recall that  the standard dyadic grid in $\mathbb{R}^n$ consists of all cubes of the form $$2^{-k}([0,\,1)^n+j),\,k\in  \mathbb{Z},\,\,j\in\mathbb{Z}^n.$$
Denote the standard grid by $\mathcal{D}$. For a fixed cube $Q$, denote by $\mathcal{D}(Q)$ the set of dyadic cubes with respect to $Q$, that is, the cubes from $\mathcal{D}(Q)$ are formed by repeating subdivision of $Q$ and each of descendants into $2^n$ congruent subcubes.

As usual, by a general dyadic grid $\mathscr{D}$,  we mean a collection of cube with the following properties: (i) for any cube $Q\in \mathscr{D}$, its side length $\ell(Q)$ is of the form $2^k$ for some $k\in \mathbb{Z}$; (ii) for any cubes $Q_1,\,Q_2\in \mathscr{D}$, $Q_1\cap Q_2\in\{Q_1,\,Q_2,\,\emptyset\}$; (iii) for each $k\in \mathbb{Z}$, the cubes of side length $2^k$ form a partition of $\mathbb{R}^n$.

Let $\eta\in (0,\,1)$ and $\mathcal{S}=\{Q_j\}$ be a family of cubes. We say that $\mathcal{S}$ is $\eta$-sparse,  if for each fixed $Q\in \mathcal{S}$, there exists a measurable subset $E_Q\subset Q$, such that $|E_Q|\geq \eta|Q|$ and $E_{Q}$'s are pairwise disjoint.
Associated with  the sparse family $\mathcal{S}$ and $\beta\in [0,\,\infty)$, define the sparse operator $\mathcal{A}_{\mathcal{S}, L(\log L)^{\beta}}$ by
$$\mathcal{A}_{\mathcal{S},\,L(\log L)^{\beta}}f(x)=\sum_{Q\in\mathcal{S}}\|f\|_{L(\log L)^{\beta},\,Q}\chi_{Q}(x).$$
It was proved in \cite{hu} that for $p\in (1,\,\infty)$, $\epsilon\in (0,\,1]$ and weight $u$,
\begin{eqnarray}\label{eq3.1}\|\mathcal{A}_{\mathcal{S},\,L(\log L)^{\beta}}g\|_{L^p(\mathbb{R}^n,\,u)}\lesssim p'^{1+\beta}p^2\big(\frac{1}{\epsilon}\big)^{\frac{1}{p'}} \|g\|_{L^p(\mathbb{R}^n,\,M_{L(\log L)^{p-1+\epsilon}}u)}.\end{eqnarray}
Moreover, for any $p\in (1,\,\infty)$ and $w\in A_{p}(\mathbb{R}^n)$,
\begin{eqnarray}\label{en3.2}\|\mathcal{A}_{\mathcal{S},L(\log L)^{\beta}}f\|_{L^p(\mathbb{R}^n,\,w)}\lesssim [w]_{A_p}^{\frac{1}{p}}\big([w]_{A_{\infty}}^{\frac{1}{p'}}+[\sigma]_{A_{\infty}}^{\frac{1}{p}}\big)[\sigma]_{A_{\infty}}^{\beta}
\|f\|_{L^p(\mathbb{R}^n,\,w)}.\end{eqnarray}

For  sparse family $\mathcal{S}$ and constants $\beta_1$, $\beta_2\in[0,\,\infty)$, we define the bi-sublinear sparse operator $\mathcal{A}_{\mathcal{S};\,L(\log L)^{\beta_1},\,L(\log L)^{\beta_2}}$  by
$$\mathcal{A}_{\mathcal{S};\,L(\log L)^{\beta_1},L(\log L)^{\beta_2}}(f,g)=\sum_{Q\in\mathcal{S}}|Q|\|f\|_{L(\log L)^{\beta_1},\,Q}\|g\|_{L(\log L)^{\beta_2},\,Q}.$$
We denote $\mathcal{A}_{\mathcal{S};\,L(\log L)^1,L(\log L)^{\beta_2}}$ by $\mathcal{A}_{\mathcal{S};\,L\log L,L(\log L)^{\beta_2}}$,
and  $\mathcal{A}_{\mathcal{S};\,L(\log L)^{\beta_1},L(\log L)^{0}}$ by $\mathcal{A}_{\mathcal{S};\,L(\log L)^{\beta_1},L}$.
The following lemma is a weighted version of Proposition 6 in \cite{bb}.
\begin{lemma}\label{lem3.2}Let  $\beta_1,\,\beta_2\in \mathbb{N}\cup\{0\}$ and  $U$ be a  linear operator. Suppose that for any bounded function $f$ with compact support, there exists a sparse family of cubes $\mathcal{S}$, such that for any function $g\in L^1(\mathbb{R}^n)$,
\begin{eqnarray}\label{en.sparse}\Big|\int_{\mathbb{R}^n}Uf(x)g(x)dx\Big|\leq \mathcal{A}_{\mathcal{S};\,L(\log L)^{\beta_1},\,L(\log L)^{\beta_2}}(f,\,g).\end{eqnarray}
Then for any $\epsilon\in (0,\,1)$ and weight $u$,
\begin{eqnarray*}&&u(\{x\in\mathbb{R}^n:\, |Uf(x)|>\lambda\})\\
&&\quad\lesssim \frac{1}{\epsilon^{1+\beta_1}}
\int_{\mathbb{R}^n}\frac{|f(x)|}{\lambda}\log ^{\beta_1}\Big({\rm e}+\frac{|f(x)|}{\lambda}\Big)M_{L(\log L)^{\beta_2+\epsilon}}u(x)dx.\nonumber\end{eqnarray*}
\end{lemma}
\begin{proof}Let $f$ be a bounded function with compact support, and $\mathcal{S}$ be the sparse family such that (\ref{en.sparse}) holds true. By the one-third trick (see \cite[Lemma 2.5]{hlp}), we may assume that $\mathcal{S}\subset \mathscr{D}$ with $\mathscr{D}$ a dyadic grid. Let $\epsilon\in (0,\,1)$ and $u$ be a weight. It was proved in \cite[pp. 618-619]{hp2} that
\begin{eqnarray*}&&
\|Mg\|_{L^{p'}(\mathbb{R}^n,\,( M_{L(\log L)^{p-1+\epsilon}}u)^{1-p'})}\lesssim p^2\big(\frac{1}{\epsilon}\big)^{\frac{1}{p'}}\|g\|_{L^{p'}(\mathbb{R}^n,u^{1-p'})}.
\end{eqnarray*}
Repeating the argument in the proof of Theorem 1.8 in \cite{lpr}, we know that for if  $p\in (1,\,\infty)$ and 
$\|f\|_{L^{p'}(\mathbb{R}^n,\,M_{L(\log L)^{(\beta_2+1)p-1+\epsilon}}u)}=1$, then
\begin{eqnarray*}
&&\mathcal{A}_{\mathcal{S};\,L(\log L)^{\beta_1},\,L(\log L)^{\beta_2}}(f,\,g)\\
&&\quad\lesssim p'^{1+\beta_1}\|M_{L(\log L)^{\beta_2}}g\|_{L^{p'}(\mathbb{R}^n,\big(M_{L(\log L)^{(\beta_2+1)p-1+\epsilon}}u\big)^{1-p'})}\\
&&\quad\lesssim p'^{1+\beta_1}[p^2\big(\frac{1}{\epsilon}\big)^{\frac{1}{p'}}]^{\beta_2+1}\|g\|_{L^{p'}(\mathbb{R}^n,\,u^{1-p'})},
\end{eqnarray*}
since $M_{L(\log L)^{\beta_2}}g(x)\approx M^{\beta_2+1}g(x)$, see \cite{cana}.
This, via homogeneity, states that
\begin{eqnarray}\label{eq3.12new}
&&\mathcal{A}_{\mathcal{S};\,L(\log L)^{\beta_1},\,L(\log L)^{\beta_2}}(f,\,g)\\
&&\quad\lesssim p'^{1+\beta_1}[p^2\big(\frac{1}{\epsilon}\big)^{\frac{1}{p'}}]^{\beta_2+1}\|f\|_{L^{p'}(\mathbb{R}^n,\,M_{L(\log L)^{(\beta_2+1)p-1+\epsilon}}u)}\|g\|_{L^{p'}(\mathbb{R}^n,\,u^{1-p'})}.\nonumber
\end{eqnarray}

Now let  $M_{\mathscr{D}, L(\log L)^{\beta_1}}$ be the maximal operator defined by
$$M_{\mathscr{D},L(\log L)^{\beta_1}}f(x)=\sup_{Q\ni x,\,Q\in\mathscr{D}}\|f\|_{L(\log L)^{\beta_1},\,Q}.$$
Decompose the set $\{x\in\mathbb{R}^n:\,M_{\mathscr{D},L(\log L)^{\beta_1}}f(x)>1\}$ as
$$\{x\in\mathbb{R}^n:\,M_{\mathscr{D},L(\log L)^{\beta_1}}f(x)>1\}=\cup_{j}Q_j,$$
with $Q_j$ the maximal cubes in $\mathscr{D}$ such that $\|f\|_{L(\log L)^{\beta_1},\,Q_j}>1$. We have that
$$1 <\|f\|_{L(\log L)^{\beta_1},\,Q_j}\lesssim 2^n.$$Let
$$f_1(y) = f(y)\chi_{\mathbb{R}^n\backslash \cup_jQ_j}(y),\,\, f_2(y) =
\sum_j
f(y)\chi_{Q_j}(y),$$
and
$$f_3(y) =\sum_j\|f\|_{L(\log L)^{\beta_1},\,Q_j}\chi_{Q_j}(y).$$
It is obvious that $\|f_1\|_{L^{\infty}(\mathbb{R}^n)}\lesssim 1$ and $\|f_3\|_{L^{\infty}(\mathbb{R}^n)}\lesssim 1$. Take $p_1=1+\frac{\epsilon}{2(\beta_2+1)}$. It then follows from the inequality (\ref{eq3.12new}) that
\begin{eqnarray}\label{en3.4}
&&\mathcal{A}_{\mathcal{S};\,L(\log L)^{\beta_1},\,L(\log L)^{\beta_2}}(f_1,\,g)\\
&&\quad\lesssim p_1'^{1+\beta_1}[p_1^2\big(\frac{1}{\epsilon}\big)^{\frac{1}{p_1'}}]^{\beta_2+1}\|f_1\|_{L^{p_1}(\mathbb{R}^n,\,M_{L(\log L)^{(\beta_2+1)p_1-1+\epsilon}}u)}\|g\|_{L^{p_1'}(\mathbb{R}^n,\,u^{1-p'})}\nonumber\\
&&\quad\lesssim\frac{1}{\epsilon^{1+\beta_1}}\|f\|^{\frac{1}{p_1}}_{L^{1}(\mathbb{R}^n,\,M_{L(\log L)^{\beta_2+\epsilon}}u)}\|g\|_{L^{p_1'}(\mathbb{R}^n,\,u^{1-p'})}.\nonumber
\end{eqnarray}

Let $E=\cup_j4nQ_j$ and   $\widetilde{u}(y)=u(y)\chi_{\mathbb{R}^n\backslash E}(y).$ It is obvious that
\begin{eqnarray}\label{en3.5}u(E)\lesssim \sum_{j}\inf_{z\in Q_j}Mu(z)|Q_j|\lesssim \int_{\mathbb{R}^n}|f(y)|\log^{\beta_1}({\rm e}+|f(y)|)Mu(y)dy.
\end{eqnarray}
Moreover, by fact that $$\inf_{y\in Q_j}M_{L(\log L)^{\gamma}}\widetilde{u}(y)\approx\sup_{z\in Q_j}M_{L(\log L)^{\gamma}}\widetilde{u}(z),$$
we obtain that for $\gamma\in [0,\,\infty)$,
\begin{eqnarray}\label{en3.6}
\|f_3\|_{L^1(\mathbb{R}^n,\,M_{L(\log L)^{\gamma}}\widetilde{u})}
&\lesssim &\sum_{j}\inf_{z\in Q_j}M_{L(\log L)^{\gamma}}\widetilde{u}(z)|Q_j|\|f\|_{L(\log L)^{\beta_1},\,Q_j}\\
&\lesssim &\int_{\mathbb{R}^n}|f(y)|\log^{\beta_1} ({\rm e}+|f(y)|)M_{L(\log L)^{\gamma}}u(y){\rm d}y.\nonumber
\end{eqnarray}
Let $$\mathcal{S}^*=\{I\in\mathcal{S}:\, I\cap(\mathbb{R}^n\backslash E)\not =\emptyset\}.$$
Note that if ${\rm supp}\,g\subset \mathbb{R}^n\backslash E$, then
$$\mathcal{A}_{\mathcal{S}, L(\log L)^{\beta_1},\,L(\log L)^{\beta_2}}(f_2,\,g)=\mathcal{A}_{\mathcal{S}^*, L(\log L)^{\beta_1},\,L(\log L)^{\beta_2}}(f_2,\,g).$$
As in the argument in \cite[pp. 160-161]{hu}, we can verify that for each fixed $I\in \mathcal{S}^*$,
$$
\|f_2\|_{L(\log L)^{\beta_1},\,I}\lesssim \|f_3\|_{L(\log L)^{\beta_1},\,I}.
$$
Therefore, for $g\in L^1(\mathbb{R}^n)$ with ${\rm supp}\,g\subset \mathbb{R}^n\backslash E$,
\begin{eqnarray}\label{en3.7}&&\mathcal{A}_{\mathcal{S},L(\log L)^{\beta_1},\,L(\log L)^{\beta_2} }(f_2,\,g)\lesssim \mathcal{A}_{\mathcal{S}, L(\log L)^{\beta_1},\,L(\log L)^{\beta_2}}(f_3,\,g)\\
&&\qquad\lesssim\frac{1}{\epsilon^{1+\beta_1}}\|f_3\|^{\frac{1}{p_1}}_{L^{1}(\mathbb{R}^n,\,M_{L(\log L)^{\beta_2+\epsilon}}u)}\|g\|_{L^{p_1'}(\mathbb{R}^n,\,u^{1-p'})}.\nonumber
\end{eqnarray}
Inequalities (\ref{en3.4}) and (\ref{en3.7}) tells us that
\begin{eqnarray*}
&&\sup_{\|g\|_{L^{p_1'}(\mathbb{R}^n\backslash E, \widetilde{u}^{1-p_1'})}\leq 1}\Big|\int_{\mathbb{R}^n}Uf(x)g(x)dx\Big|\\
&&\quad\lesssim\sup_{\|g\|_{L^{p_1'}(\mathbb{R}^n\backslash E, \widetilde{u}^{1-p_1'})}\leq 1}\Big(\mathcal{A}_{\mathcal{S},\,L(\log L)^{\beta_1},L(\log L)^{\beta_2}}(f_1,g)+\mathcal{A}_{\mathcal{S},\,L(\log L)^{\beta_1},L(\log L)^{\beta_2}}(f_2,g)\\
&&\quad\lesssim\frac{1}{\epsilon^{1+\beta_1}}\Big(\|f\|^{\frac{1}{p_1}}_{L^{1}(\mathbb{R}^n,\,M_{L(\log L)^{\beta_2+\epsilon}}\widetilde{u})}+\|f_3\|^{\frac{1}{p_1}}_{L^{1}(\mathbb{R}^n,\,M_{L(\log L)^{\beta_2+\epsilon}}u)}\Big).
\end{eqnarray*}
Thus by  a duality argument,
\begin{eqnarray*}
&&u(\{x\in\mathbb{R}^n:\,|Uf(x)|>1\})\leq u(E)+\|Uf\|_{L^{p_1}(\mathbb{R}^n\backslash E,\,\widetilde{u})}^{p_1}\\
&&\quad\lesssim \frac{1}{\epsilon^{1+\beta_1}}\int_{\mathbb{R}^n}|f(y)|\log^{\beta_1}({\rm e}+|f(y)|)M_{L(\log L)^{\beta_2+\epsilon}}u(y)dy.\nonumber
\end{eqnarray*}
This, via homogeneity, leads to our desired conclusion.
\end{proof}
\begin{theorem}\label{thm3.1}Let $T_1$ and $T_2$ be Calder\'on-Zygmund operators, $b\in{\rm BMO}(\mathbb{R}^n)$ with $\|b\|_{{\rm BMO}(\mathbb{R}^n)}=1$.
\begin{itemize}
\item[\rm (i)] For  bounded function $f$ with compact support, there exists a $\frac{1}{2}\frac{1}{9^n}$-sparse family of cubes $\mathcal{S}=\{Q\}$, and functions $J_0$, $J_1$, such that for each $j=0,\,1$ and function $g$,
$$\Big|\int_{\mathbb{R}^n}J_j(x)g(x)dx\Big|\lesssim \mathcal{A}_{\mathcal{S};\,L(\log L)^{1-j},\,L(\log L)^j}(f,\,g);$$
and for a. e. $x\in\mathbb{R}^n$,
$$T_{1} T_2f(x)=J_0(x)+J_1(x).
$$
\item [\rm (ii)] For  bounded function $f$ with compact support, there exists a $\frac{1}{2}\frac{1}{9^n}$-sparse family of cubes $\mathcal{S}^*=\{Q\}$, and functions $H_0$, $H_1$ and $H_2$, such that for each $j=0,\,1,\,2$ and function $g$,
$$\Big|\int_{\mathbb{R}^n}H_j(x)g(x)dx\Big|\lesssim \mathcal{A}_{\mathcal{S};\,L(\log L)^{2-j},\,L(\log L)^j}(f,\,g);$$
and for a. e. $x\in\mathbb{R}^n$,
$$T_{1,b} T_2f(x)=H_0(x)+H_1(x)+H_2(x).
$$
\end{itemize}
\end{theorem}
\begin{proof}We only prove the conclusion (ii). The proof of conclusion (i) is similar, more simpler and will be omitted. We will employ the argument in \cite{ler2}. For a fixed cube $Q_0$, define the  local analogy of $\mathcal{M}_{T_2}$, $\mathcal{M}^*_{T_{1}T_2}$ and $\mathcal{M}^*_{T_{1}T_{2,b}}$ by
$$ \mathcal{M}_{T_2,\,Q_0}f(x)=\sup_{Q\ni x,\, Q\subset Q_0}{\rm ess}\sup_{\xi\in Q}|T_2(f\chi_{3Q_0\backslash 3Q})(\xi)|,$$
$$\mathcal{M}^*_{T_{1}T_2;\,Q_0}f(x)=\sup_{Q\ni x,\, Q\subset Q_0}{\rm ess}\sup_{\xi\in Q}|T_1\big(\chi_{\mathbb{R}^n\backslash 3Q}
T_2(f\chi_{9Q_0\backslash 9Q})\big)(\xi)|,$$
and
$$\mathcal{M}^*_{T_{1}T_{2,b};\,Q_0}f(x)=\sup_{Q\ni x,\, Q\subset Q_0}{\rm ess}\sup_{\xi\in Q}|T_1\big(\chi_{\mathbb{R}^n\backslash 3Q}
T_{2,b}(f\chi_{9Q_0\backslash 9Q})\big)(\xi)|,$$
respectively. Let $E=\cup_{j=1}^5E_j$ with
$$E_1=\big\{x\in Q_0:\, |T_{1,\,b}T_2(f\chi_{9Q_0})(x)|>D\|f\|_{L(\log L)^2,\,9Q_0}\big\},$$
$$E_2=\{x\in Q_0:\,\mathcal{M}_{T_2,\,Q_0}f(x)>D\langle |f|\rangle_{9Q_0}\},$$
$$E_3=\big\{x\in Q_0:\, \mathcal{M}^*_{T_{1}T_2;\,Q_0}f(x)>D\|f\|_{L\log L,9Q_0}\big\},$$and
$$E_4=\big\{x\in Q_0:\, \mathcal{M}^*_{T_{1}T_{2, b};\,Q_0}f(x)>D\|f\|_{L(\log L)^2,9Q_0}\big\},$$
$$E_5=\big\{x\in Q_0:\, \mathcal{M}^*_{T_{1}T_2;\,Q_0}\big((b-\langle b\rangle_{Q_0})f\big)(x)>D\|f\|_{L(\log L)^2,9Q_0}\big\},$$
with $D$ a positive constant. It then follows from Theorem \ref{thm2.1}, Corollary \ref{co2.1} and (\ref{eq1.10})  that
$$|E|\le \frac{1}{2^{n+2}}|Q_0|,$$
if we choose $D$ large enough. Now on the cube $Q_0$, we apply the Calder\'on-Zygmund decomposition to $\chi_{E}$ at level $\frac{1}{2^{n+1}}$, and obtain pairwise disjoint cubes $\{P_j\}\subset \mathcal{D}(Q_0)$, such that
$$\frac{1}{2^{n+1}}|P_j|\leq |P_j\cap E|\leq \frac{1}{2}|P_j|$$
and $|E\backslash\cup_jP_j|=0$.  Observe that $\sum_j|P_j|\leq \frac{1}{2}|Q_0|$.
Let
\begin{eqnarray*}
G_0(x)&=&T_{1,b} T_2(f\chi_{9Q_0})(x)\chi_{Q_0\backslash \cup_{l}P_l}(x)\\
&&-\sum_lT_{1}\Big(\chi_{\mathbb{R}^n\backslash 3P_l}T_{2,b}(f\chi_{9Q_0\backslash 9P_l})\Big)(x)\chi_{P_l}(x)\\
&&-\sum_lT_{1}\Big(\chi_{\mathbb{R}^n\backslash 3P_l}T_2\big((b-\langle b\rangle_{Q_0})f\chi_{9Q_0\backslash 9P_l}\big)\Big)(x)\chi_{P_l}(x).
\end{eqnarray*}
The facts that $P_j\cap E^c\not =\emptyset$ and $|E\backslash\cup_jP_j|=0$ imply that
\begin{eqnarray}\label{eq3.8}|G_0(x)|\lesssim \|f\|_{L(\log L)^2,\,9Q_0}.\end{eqnarray}
Also, we define functions $G_1$ and $G_2$ by
$$G_{1}(x)=\big(b(x)-\langle b\rangle_{Q_0}\big)\sum_lT_{1}\big(\chi_{\mathbb{R}^n\backslash 3P_l}T_2(f\chi_{9Q_0\backslash 9P_l})\big)(x)\chi_{P_l}(x),
$$
and
$$G_2(x)=\sum_lT_{1,\,b}\big(\chi_{3P_l}T_2\big(f\chi_{9Q_0\backslash 9P_l})\big)(x)\chi_{P_l}(x).
$$
Then
\begin{eqnarray}\label{eq3.9}|G_{1}(x)|\lesssim |b(x)-\langle b\rangle_{Q_0}|\|f\|_{L\log L,\,9Q_0}\chi_{Q_0}(x).
\end{eqnarray}
Let $\widetilde{T}_1$ be the adjoint operator of $T_1$ and $\widetilde{T}_{1,b}$ the commutator of $\widetilde{T}_1$. For each function $g$, we have by Lemma \ref{lem2.1} that
\begin{eqnarray}\label{eq3.10}
\Big|\int_{\mathbb{R}^n}G_{2}(x)g(x)dx\Big|&\le &\sum_l\int_{3P_l}\big|T_2\big(f\chi_{9Q_0\backslash 9P_l}\big)(x)\widetilde{T}_{1,b}(g\chi_{P_l})(x)\big|dx\\
&\lesssim&\sum_l\inf_{\xi\in P_l}\mathcal{M}_{T_2,\,Q_0}f(\xi)\int_{3P_l}|\widetilde{T}_{1,b}(g\chi_{P_l})(x)|dx\nonumber\\
&\lesssim&\langle |f|\rangle_{9Q_0}\|g\|_{L(\log L)^2,\,Q_0}|Q_0|.\nonumber
\end{eqnarray}
Moreover,
$$T_{1,b} T_2(f\chi_{9Q_0})(x)\chi_{Q_0}(x)=G_0(x)+G_1(x)+G_2(x)+\sum_{l}T_{1,b}T_2(\chi_{9P_l})(x)\chi_{P_l}(x).$$

We now repeat the argument above with $T_{1,b}T_2(f\chi_{9Q_0})(x)\chi_{Q_0}$ replaced by each $T_{1,b}T_2(\chi_{9P_l})(x)\chi_{P_l}(x)$, and so on.
Let $Q_{0}^{j_1}=P_{j}$,  and for fixed $j_1,\,\dots,\,j_{m-1}$, $\{Q_{0}^{j_1...j_{m-1}j_m}\}_{j_m}$ be the cubes obtained at the $m$-th stage of the decomposition process to the cube $Q_{0}^{j_1...j_{m-1}}$.
For each fixed $j_1\dots,j_m$, define the functions $H_{Q_0}^{j_1\dots j_m}f$, $H_{Q_0,1}^{j_1\dots j_m}f$ and  $H_{Q_0,2}^{j_1\dots j_m}f$ by
\begin{eqnarray*}
&&H_{Q_0,0}^{j_1\dots j_m}f(x)=-T_{1}\big(\chi_{\mathbb{R}^n\backslash 3Q_0^{j_1\dots j_m}}T_{2,b}(f\chi_{9Q_0^{j_1\dots j_{m-1}}\backslash 9Q_0^{j_1\dots j_m}})\big)(x)\chi_{Q_0^{j_1\dots j_m}}(x)\\
&&\qquad-T_{1}\Big(\chi_{\mathbb{R}^n\backslash 3Q_0^{j_1\dots j_m}}T_2\big((b-\langle b\rangle_{Q_0})f\chi_{9Q_0^{j_1\dots j_{m-1}}\backslash 9Q_0^{j_1\dots j_m}}\big)\Big)(x)\chi_{3Q_0^{j_1\dots j_m}}(x),
\end{eqnarray*}
\begin{eqnarray*}H_{Q_0,1}^{j_1\dots j_m}f(x)&=&\big(b(x)-\langle b\rangle_{Q_{0}^{j_1\dots j_{m-1}}}\big)\\
&&T_1\Big(\chi_{3Q_{0}^{j_1\dots j_m}}T_2(f\chi_{9Q_{0}^{j_1\dots j_{m-1}}\backslash 9Q_{0}^{j_1\dots j_m}})\big)\Big)(x)\chi_{Q_{0}^{j_1\dots j_m}}(x),
\end{eqnarray*}
and
$$H_{Q_0,2}^{j_1\dots j_m}f(x)=T_{1,b}\Big(\chi_{3Q_{0}^{j_1\dots j_m}}T_2(f\chi_{9Q_{0}^{j_1\dots j_{m-1}}\backslash 9Q_{0}^{j_1\dots j_m}})\big)\Big)(x)\chi_{Q_{0}^{j_1\dots j_m}}(x),
$$ respectively.
Set $\mathcal{F}=\{Q_0\}\cup_{m=1}^{\infty}\cup_{j_1,\dots,j_m}\{Q_{0}^{j_1\dots j_m}\}$. Then $\mathcal{F}\subset \mathcal{D}(Q_0)$ be a $\frac{1}{2}$-sparse  family.  Let
\begin{eqnarray*}&&H_{0,\,Q_0}(x)=T_{1,b}T_2(f\chi_{9Q_0})\chi_{Q_0\backslash \cup_{j_1}Q_{0}^{j_1}}(x)\\
&&\quad+\sum_{m=1}^{\infty}\sum_{j_1,\dots,j_m}T_{1,b}T_2(f\chi_{9Q_{0}^{j_1\dots j_m}})\chi_{Q_{0}^{j_1\dots j_m}\backslash \cup_{j_{m+1}}Q_{0}^{j_1\dots j_{m+1}}}(x)\\
&&\quad+\sum_{m=1}^{\infty}\sum_{j_1\dots j_m}H_{Q_0}^{j_1\dots j_m}f(x)\chi_{Q_{0}^{j_1\dots j_m}}(x),
\end{eqnarray*}
Also, we define the functions $H_{1,\,Q_0}$ and $H_{2,\,Q_0}$ by
\begin{eqnarray*}H_{1,Q_0}(x)&=&\sum_{m=1}^{\infty}\sum_{j_1\dots j_m}H_{Q_0,1}^{j_1\dots j_m}f(x)\chi_{Q_{0}^{j_1\dots j_m}}(x),
\end{eqnarray*}
and
$$H_{2,Q_0}(x)=\sum_{m=1}^{\infty}\sum_{j_1\dots j_m}H_{Q_0,2}^{j_1\dots j_m}f(x)\chi_{Q_{0}^{j_1\dots j_m}}(x).$$
Then for a. e. $x\in Q_0$,
$$T_{1,b} T_2(f\chi_{9Q_0})(x)=H_{0,Q_0}f(x)+H_{1,Q_0}(x)+H_{2,Q_0}(x).
$$
Moreover, as in the inequalities (\ref{eq3.8})-(\ref{eq3.10}), the process of producing $\{Q_0^{j_1\dots j_m}\}$ tells us that
$$|H_{0,Q_0}f(x)\chi_{Q_0}|\lesssim \sum_{Q\in\mathcal{F}}\|f\|_{L(\log L)^2,\,9Q}\chi_{Q}(x).$$
For any function $g$, we can verify that
$$\Big|\int_{Q_0}g(x)H_{1,Q_0}(x)dx\Big|\lesssim \sum_{Q\in \mathcal{F}}|Q|\|f\|_{L\log L,\,9Q}\|g\|_{L\log L,\,Q},$$
and
$$\Big|\int_{Q_0}g(x)H_{2,Q_0}(x)dx\Big|\lesssim \sum_{Q\in \mathcal{F}}|Q|\langle |f|\rangle_{9Q}\|g\|_{L(\log L)^2,\,Q}.$$

We can now conclude the proof of Theorem \ref{thm3.1}. In fact, as in \cite{ler3}, we decompose $\mathbb{R}^n$ by cubes $R_l$, such that ${\rm supp}f\subset 3R_l$ for each $l$, and $R_l$'s have disjoint interiors.
Then for a. e. $x\in\mathbb{R}^n$,
\begin{eqnarray*}T_{1,\,b} T_2f(x)&=&\sum_{l}H_{0,R_l}f(x)+\sum_lH_{1,R_l}f(x)+\sum_lH_{2, R_l}f(x)\\
&:=&H_0f(x)+H_1f(x)+H_2f(x).\end{eqnarray*}
Obviously, the functions $H_0$, $H_1$ and $H_2$ satisfies conclusion (i). Our desired conclusion then follows directly.
\end{proof}
\begin{remark}By Theorem \ref{thm3.1}, we see that for bounded function $f$ with compact support, there exists a $\frac{1}{2}\frac{1}{9^n}$-sparse family $\mathcal{S}$, such that for any $g$,
$$
\Big|\int_{\mathbb{R}^n}g(x)T_1T_2f(x)dx\Big|\lesssim \mathcal{A}_{\mathcal{S}; L\log L, L^1}(f,\,g)+\mathcal{A}_{\mathcal{S}; L^1, L\log L}(f,\,g),
$$
and
$$
\Big|\int_{\mathbb{R}^n}g(x)T_{1,\,b}T_2f(x)dx\Big|\lesssim \sum_{j=0}^2\mathcal{A}_{\mathcal{S}; L(\log L)^{2-j},L(\log L)^j}(f,\,g).
$$
\end{remark}

\medskip

{\it Proof of Theorem \ref{thm1.1}}. By Theorem \ref{thm3.1} and Lemma \ref{lem3.2}, we know that for each $\epsilon\in (0,\,1)$, weight $u$ and $\lambda>0$,
\begin{eqnarray*}
u\big(\{x\in\mathbb{R}^n:|T_1T_2f(x)|>\lambda\}\big)&\lesssim &\int_{\mathbb{R}^n}\frac{|f(x)|}{\lambda}\log\Big({\rm e}+\frac{|f(x)|}{\lambda}\Big)M_{L\log L}u(x)dx\\
&&+\frac{1}{\lambda\epsilon}\int_{\mathbb{R}^n}|f(x)|M_{L(\log L)^{1+\epsilon}}u(x)dx\\
&\lesssim&\frac{1}{\epsilon}\int_{\mathbb{R}^n}\frac{|f(x)|}{\lambda}\log\big({\rm e}+\frac{|f(x)|}{\lambda}\big)M_{L(\log L)^{1+\epsilon}}u(x)dx.
\end{eqnarray*}
Applying the ideas used in \cite[p. 608]{hp2} (see also the proof of Corollary 1.3 in \cite{ler4}), we know that the last inequality implies (\ref{eq1.8}).

The inequality (\ref{eq1.81})  is essentially an application of Proposition 9 in \cite{bb}. Recall that $T_1(1)=0$. It then follows from \cite[Proposition 9]{bb} that for $f\in L^2(\mathbb{R}^n)$, there exists a sparse family of cubes $\mathcal{S}$, such that
$$|T_1T_2f(x)|\lesssim \sum_{Q\in\mathcal{S}}{\rm osc}_Q(T_2f)\chi_{Q}(x),
$$
here ${\rm osc}_Q(T_2f)$ is defined by
$${\rm osc}_Q(T_2f)=\frac{1}{|Q|}\int_{Q}\big|T_2f(x)-\langle T_2f\rangle_Q\big|dx.$$
A trivial computation leads to that
\begin{eqnarray*}
{\rm osc}_Q(T_2f)\lesssim \|f\|_{L\log L,\,8nQ}+\sum_{k=1}^{\infty}2^{-k\varepsilon}\langle |f|\rangle_{2^kQ},
\end{eqnarray*}with $\varepsilon$ the constant in (\ref{eq1.verepsilon}).
Let $G$ be the operator defined by $$Gf(x)=\sum_{k=1}^{\infty}2^{-k\varepsilon}\sum_{Q\in \mathcal{S}}\langle |f|\rangle_{2^kQ}\chi_{Q}(x).$$
We then have that
\begin{eqnarray}\label{en3.11}|T_1T_2f(x)|\lesssim \mathcal{A}_{\mathcal{S},L\log L}f(x)+Gf(x).
\end{eqnarray}
On the other hand, it was proved in \cite{ler1} that, there exist  sparse family of cubes $\mathcal{S}_1,\,\dots,\,\mathcal{S}_{2^n+1}$
such that for any function $g$,
\begin{eqnarray}\label{en3.12}
\int_{\mathbb{R}^n}|Gf(x)g(x)|dx\lesssim\sum_{j=1}^{2^n+1}\mathcal{A}_{\mathcal{S}_j,L,\,L}(f,\,g).
\end{eqnarray}
Thus, by  Lemma \ref{lem3.2}, we know that for each fixed $\lambda>0$, $\epsilon\in (0,\,1)$ and weight $u$,
\begin{eqnarray*}
&&u\big(\{x\in\mathbb{R}^n:\, |T_{1} T_2f(x)|>\lambda\}\big)\\
&&\quad\lesssim \frac{1}{\epsilon^2}\int_{\mathbb{R}^n}\frac{|f(x)|}{\lambda}\log\Big({\rm e}+\frac{|f(x)|}{\lambda}\Big)M_{L(\log L)^{\epsilon}}u(x)dx.\nonumber
\end{eqnarray*}
This implies (\ref{eq1.81}). \qed

\begin{remark}\label{r3.2} By the estimate of bilinear sparse operator (see \cite{hhl} or \cite{li2}) , we know that for $p\in (1,\,\infty)$ and $w\in A_p(\mathbb{R}^n)$,
\begin{eqnarray}\label{en3.13}\mathcal{A}_{\mathcal{S}, L^1,\,L^1}(f,\,g)\lesssim [w]_{A_p}^{\frac{1}{p}} \big([w]_{A_{\infty}}^{\frac{1}{p'}}+[\sigma]_{A_{\infty}}^{\frac{1}{p}}\big)\|f\|_{L^p(\mathbb{R}^n,\,w)}\|g\|_{L^{p'}(\mathbb{R}^n,\,\sigma)}.
\end{eqnarray}
The conclusions in Theorem \ref{thm1.0} now follows from (\ref{en3.11})-(\ref{en3.13}) and   inequality (\ref{en3.2}). Moreover, by (\ref{en3.11})-(\ref{en3.12}), (\ref{eq3.12new}) and (\ref{eq3.1}), we know that for $p\in (1,\infty)$, $\epsilon\in (0,\,1)$ and weight $u$,
$$\|T_1T_2f\|_{L^p(\mathbb{R}^n,\,u)}\lesssim p'^{2}p^2\big(\frac{1}{\epsilon}\big)^{\frac{1}{p'}} \|f\|_{L^p(\mathbb{R}^n,\,M_{L(\log L)^{p-1+\epsilon}}u)}.$$
\end{remark}
\begin{remark}\label{r3.3} Let $T_1,\,\dots\,T_m$ be Calder\'on-Zygmund operators. Repeating the proof of Theorem \ref{thm1.1}, we can verify that for each bounded function $f$, there exists a $\frac{1}{2}\frac{1}{3^m}$-sparse family of cubes $\mathcal{S}$, and  functions $J_0,\,\dots,\,J_{m_1}$ such that for each $j=0,\,\dots, m-1$, and function $g\in L^1(\mathbb{R}^n)$,
$$\Big|\int_{\mathbb{R}^n}J_j(x)g(x)dx\Big|\lesssim \mathcal{A}_{\mathcal{S},\,L(\log L)^{m-j},\,L(\log L)^j}(f,\,g),$$
and for a. e. $x\in \mathbb{R}^n$,
$$T_1\dots T_mf(x)=\sum_{j=0}^{m-1}J_j(x).$$
Moreover, for each $\epsilon\in (0,\,1)$, weight $u$ and $\lambda>0$,
\begin{eqnarray*}
&&u\big(\{x\in\mathbb{R}^n:|T_1\dots T_mf(x)|>\lambda\}\big)\\
&&\quad\lesssim\frac{1}{\epsilon}\int_{\mathbb{R}^n}\frac{|f(x)|}{\lambda}\log^{m-1}\Big({\rm e}+\frac{|f(x)|}{\lambda}\Big)M_{L(\log L)^{m-1+\epsilon}}u(x)dx.
\end{eqnarray*}
\end{remark}

{\it Proof  of Theorem \ref{thm1.2}}. As it was shown in the proof of Theorem \ref{thm1.1}, by Theorem \ref{thm3.1} and Lemma \ref{lem3.2}, we know that for each $\epsilon\in (0,\,1)$, weight $u$ and $\lambda>0$,
\begin{eqnarray*}
&&u\big(\{x\in\mathbb{R}^n:|T_{1,\,b}T_2f(x)|>\lambda\}\big)\\
&&\quad\lesssim\frac{1}{\epsilon}\int_{\mathbb{R}^n}\frac{|f(x)|}{\lambda}\log^2\big({\rm e}+\frac{|f(x)|}{\lambda}\big)M_{L(\log L)^{2+\epsilon}}u(x)dx.
\end{eqnarray*}
The conclusion then follows immediately.\qed

\end{document}